\documentclass[11pt]{amsart}
\usepackage{amssymb}
\newcommand{\R}{{\mathbb R}}
\newcommand{\I}{{\mathbf I}}
\newcommand{\ca}{\mathcal}

\newcommand{\dy}{\mathcal{D}}
\newcommand{\Ac}{\mathcal{A}}
\newcommand{\vp}{\varphi}
\newcommand{\Cnot}{C_0^{\infty}(\R^n)}
\newcommand{\w}{\mathbf{W}_{1,p}}
\newcommand{\wa}{\mathbf{W}_{\alpha,p}}
\newcommand{\re}{\mathbf{I}_1}
\begin{document}
\title[Finite energy solutions of quasilinear equations]{Finite energy solutions 
of\\ quasilinear elliptic equations\\ with sub-natural growth terms}
\author{Cao Tien Dat}
\address{Department of Mathematics, University of Missouri, Columbia, Missouri 65211, USA}
\email{dtcznb@mail.missouri.edu}
\author{ Igor E. Verbitsky}
\email{verbitskyi@missouri.edu}
\thanks{Supported in part by
NSF grant DMS-1161622.}
\begin{abstract}
We study finite energy solutions to quasilinear elliptic equations of the type  
$$
-\Delta_pu=\sigma \, u^q \quad \text{in ~~} \mathbb{R}^n,$$
where $\Delta_p$ is the $p$-Laplacian,  $p>1$, 
and $\sigma$ is a nonnegative function (or measure) on $\mathbb{R}^n$, in the case $0<q < p-1$ (below the ``natural growth'' rate $q=p-1$). 
We give an explicit necessary and sufficient condition on $\sigma$ which ensures that there exists a solution $u$ in the 
homogeneous Sobolev 
space $L_0^{1,p}(\mathbb{R}^n)$, and prove its uniqueness. Among our main tools are integral inequalities closely associated with this problem, and Wolff potential estimates used to obtain sharp bounds of solutions. More general quasilinear equations with 
the $\mathcal{A}$-Laplacian $ \text{div} \mathcal{A}(x,\nabla \cdot)$ in place of $\Delta_p$ are considered as well. 
 \end{abstract}

\maketitle
\numberwithin{equation}{section}
\newtheorem{theorem}{Theorem}[section]
\newtheorem{lemma}[theorem]{Lemma}
\newtheorem{rmk}[theorem]{Remark}
\newtheorem{cor}[theorem]{Corollary}
\newtheorem{prop}[theorem]{Proposition}
\newtheorem{defn}[theorem]{Definition}
\newtheorem*{theorem*}{Theorem}
\allowdisplaybreaks
\section{Introduction} \label{intro}

This paper is concerned with quasilinear problems of the following type:  
\begin{equation}\label{eq1}
-\Delta_pu=\sigma \, u^q \quad \text{in ~~} \mathbb{R}^n,
\end{equation}
where  $\Delta_p u=\nabla \cdot(\nabla u |\nabla u|^{p-2})$ is the $p$-Laplacian, $1<p< \infty$, and $\sigma$ is a nonnegative function, or measure,  in the \textit{sub-natural growth} case  $0<q<p-1$.  We are interested 
in finite energy  solutions $u\in L_0^{1, p}(\mathbb{R}^n)$ to \eqref{eq1}, and related integral inequalities. Here  $L_0^{1, p}(\mathbb{R}^n)$ is the homogeneous Sobolev (or Dirichlet) space defined in Sec.~\ref{preli} (see  \cite{HKM},  \cite{MZ}, \cite{Maz}); for $1<p<n$ it can be identified  with the 
completion of $C^\infty_0(\R^n)$  in the norm 
\begin{equation}\label{seminorm}
||u||_{1, p} = \Big( \int_{\R^n} | \nabla u(x)|^p \, dx \Big)^{\frac 1 p}. 
\end{equation}

More precisely,  $u$ is called a finite energy solution to \eqref{eq1} if 
$u \in L_0^{1, p} (\R^n)\cap L^q_{\text{loc}}(\R^n, d \sigma)$, $u \ge 0$, and, for all $\varphi \in C^\infty_0(\R^n)$, 
\begin{equation}\label{sol-def}
\int_{\R^n} |\nabla u|^{p-2}\nabla u\cdot \nabla \varphi \, dx = \int_{\R^n} u^q \, \varphi \, d \sigma.  
\end{equation}

Finite energy solutions to \eqref{eq1} are critical points of the functional 
$$
H[\varphi]=\int_{\R^n} \frac{1}{p} \, |\nabla \varphi|^p \, dx - \int_{\R^n}\frac{1}{q+1} \, |\varphi|^{1+q} \, d \sigma. 
$$
%
%
 
We will give a \textit{necessary and sufficient} condition for the existence of a finite energy solution 
to \eqref{eq1}, and prove its uniqueness. 

Our results are new even in the classical case $p=2$, $0<q<1$. Sublinear elliptic problems of this type were  
studied by Brezis and Kamin in \cite{BK}, where a necessary and sufficient 
condition is found for the existence of a \textit{bounded} solution on $\R^n$, together with sharp pointwise estimates 
of solutions.  Recently, 
we have extended these 
results to the case $p \not=2$, under relaxed assumptions on $\sigma$, in such a way that some singular (unbounded) 
solutions are covered as well \cite{CV}. However, the techniques  used in \cite{CV} are quite different from those used in this paper. 

 Analogous sublinear problems in bounded domains $\Omega \subset \mathbb{R}^n$ for  
 various classes of $\sigma$ have been extensively studied. In particular,  Boccardo and Orsina \cite{BO1}, \cite{BO2}, 
and Abdel Hamid and Bidaut-V\'{e}ron \cite{AHBV} gave sufficient conditions for the existence of solutions 
under the assumption $\sigma \in L^r (\Omega)$. Earlier results, under more restrictive assumptions on $\sigma$, can be found in 
Krasnoselskii \cite{Kr}, Brezis and Oswald \cite{BO}, and the literature cited in these papers.

 We employ powerful Wolff potential estimates developed in \cite{KM2} (see also \cite{Lab}, \cite{TW1}, \cite{KuMi}). This   
makes it possible to replace  the 
$p$-Laplacian $\Delta_p$ in the model problem \eqref{eq1} by a more general quasilinear 
operator $ \text{div} \mathcal{A}(x,\nabla \cdot)$ with bounded measurable coefficients, under standard structural assumptions on 
$\mathcal{A}(x,\xi)$ which ensure that $\mathcal{A}(x,\xi) \cdot \xi \approx  |\xi|^p \, $ 
\cite{HKM}, \cite{MZ}, or a fully nonlinear operator of $k$-Hessian type \cite{TW}, \cite{Lab} (see also  \cite{PV2}, \cite{JV2}), and treat more general nonlinearities on the right-hand side. 
Equations involving operators of the $p$-Laplacian type on Carnot groups can be covered as well using 
methods developed in \cite{PV3}.  

Wolff's potential $\mathbf{W}_{1, p}\sigma$ of a nonnegative Borel measure $\sigma$ on $\R^n$ is defined by 
\cite{HW} (see also \cite{AH}):  
\begin{equation}\label{wolffpot}
\mathbf{W}_{1, p}\sigma(x)=\int_0^{\infty} \left(\frac{\sigma(B(x,t))}{t^{n-p}}\right)^{\frac{1}{p-1}}\frac{dt}{t}.
 \end{equation}
Here $B(x,t)= \{y \in \mathbb{R}^n: \, \, |x-y|<t\}$ is a ball centered at $x \in  \mathbb{R}^n$ of radius $t>0$.

An important theorem due to Kilpel\"{a}inen and Mal\'y \cite {KM2} (see also \cite{K2})  states that if $U$ is a solution (understood in the potential theoretic or renormalized sense) to the equation 
\begin{equation}\label{Ueq}
\left\{ \begin{array}{ll}
-\Delta_p U  =\sigma & \text{in ~~} \mathbb{R}^n,\\
\displaystyle{\, \, \, \inf_{\mathbb{R}^n} \, U=0},
\end{array} \right.
\end{equation}
then there exists a constant $K>0$ which depends only on $p$ and $n$ such that 
\begin{equation}\label{Uest}
 \frac{1}{K} \, \mathbf{W}_{1, p} \sigma (x) \le U(x)  \le K  \, \mathbf{W}_{1, p} \sigma (x), \quad x \in  \mathbb{R}^n. 
\end{equation}
Moreover, $U$  exists if and only if $ \mathbf{W}_{1, p} \sigma \not\equiv +\infty$ (see \cite{PV1}), or equivalently, 
\begin{equation}\label{Unontrivial}
\int_1^{\infty} \left(\frac{\sigma(B(0,t))}{t^{n-p}}\right)^{\frac{1}{p-1}}\frac{dt}{t} < +\infty.
\end{equation}

Our main result  is the following 

\begin{theorem*}\label{energythm} Let $0<q<p-1$, $1 < p <n$, and let $\sigma$ be a locally finite positive measure on $\mathbb{R}^n$. Then there exists a nontrivial solution  $u \in L_0^{1,p}(\R^n)\cap L_{\text{loc}}^q(\Omega, d \sigma)$ to \eqref{eq1}  if and only if $U \in L^{\frac{(1+q)(p-1)}{p-1-q}}(\R^n, d \sigma)$, or equivalently, 
\begin{equation}\label{wolffintegral}
\int_{\R^n}\left(\w\sigma\right)^{\frac{(1+q)(p-1)}{p-1-q}}d\sigma < \infty.
 \end{equation}
Furthermore, such a solution is unique. For $p\ge n$, \eqref{eq1} has only a trivial solution $u=0$.
\end{theorem*}

We observe that  \eqref{wolffintegral} yields $\sigma \in L^{-1, p'}_{{\rm loc}}(\R^n)$, 
where  $L^{-1, p'} (\R^n)=L^{1, p}_0(\R^n)^*$ is the dual Sobolev space (see definitions in Sec.~\ref{preli}). 
Consequently, $\sigma$ is necessarily absolutely continuous with respect to the $p$-capacity $\text{cap}_p(\cdot)$ 
defined by 
\begin{equation} \label{pcapacity}
\text{cap}_p(E)=\inf \{||\nabla \phi||^{p}_{L^p} : \phi \ge 1 \text{ on $E$, } \phi \in C_0^{\infty}(\R^n)\},
\end{equation}
for a compact set $E \subset \R^n$. 

Moreover, as was shown in  \cite{COV1}, condition \eqref{wolffintegral} holds if and only if there exists a constant 
$C$ such that,  for all $\vp \in C^\infty_0 (\R^n)$, 
\begin{equation}\label{energy-cond}
\Big (\int_{\R^n} |\vp |^{1+q} \, d \sigma \Big)^{\frac{1}{1+q}} \le C \, || \nabla \vp ||_{L^p(\R^n)}.  
 \end{equation}
 An obvious sufficient condition which follows from Sobolev's inequality 
 is $\sigma \in L^r(\R^n)$, 
 $r=\frac{np}{n(p-1-q)+p(1+q)}$.

 There is also an equivalent characterization of \eqref {energy-cond}  in terms of capacities due to Maz'ya and Netrusov (see  
 \cite{Maz}, Sec. 11.6): 
 \begin{equation}\label{cap-cond}
\int_0^{\sigma(\R^n)} \left [ \frac {t}{\varkappa(\sigma, t)}\right]^{\frac{1+q}{p-1-q}} dt < +\infty,
 \end{equation}
where $\varkappa(\sigma, t)=\inf\{ \, \text{cap}_p (E): \, \, \sigma(E)\ge t \}$. 

Thus, any one of the conditions \eqref{wolffintegral}, \eqref{energy-cond}, and \eqref{cap-cond} is necessary 
and sufficient  for the existence of a nontrivial finite energy solution to \eqref{eq1}. \bigskip

We now outline the contents of the paper. Sec.~\ref{preli} contains definitions and notations, along with several useful results on 
quasilinear equations that will be used below. In Sec.~\ref{energy} we study the corresponding integral inequalities, deduce a necessary and 
sufficient condition for the existence of a finite energy solution, and construct a minimal solution. 
Sec.~\ref{alaplace} is devoted to more general equations with the operator $ \text{div} \mathcal{A}(x,\nabla \cdot)$  in place of the $p$-Laplacian. In Sec.~\ref{unique} 
we prove the uniqueness property of finite energy solutions.

\section{Preliminaries}\label{preli} 

We first recall some notations and definitions. Given an open set $\Omega\subseteq \R^n$, we denote by 
$M^+(\Omega)$ the class of all nonnegative Borel measures in $\Omega$ which are finite on compact subsets of $\Omega$. The $\sigma$-measure 
of a measurable set $E\subset \Omega$ is denoted by 
$
 |E|_\sigma=\sigma(E) = \int_E d \sigma.
$

For $p>0$ and $\sigma \in  M^+(\Omega)$, we denote by $L^p(\Omega, d\sigma)$ ($L_{\text{loc}}^p(\Omega, d\sigma)$, respectively) the space of measurable functions $\varphi$ such that $|\varphi|^p$ is integrable (locally integrable) with respect to $\sigma$. For $u \in L^p(\Omega, d\sigma)$, we set 
$$
||u||_{L^p(\Omega, d \sigma)} = \Big(\int_\Omega |u|^p \, d \sigma\Big)^{\frac{1}{p}}. 
$$
When $d\sigma=dx$, we write $L^p(\Omega)$ (respectively $L_{\text{loc}}^p(\Omega)$), and 
denote Lebesgue measure of $E\subset \R^n$ by $|E|$.  

The Sobolev space $W^{1, p}(\Omega)$ ($W_{\text{loc}}^{1, p}(\Omega)$, respectively) is the space of all functions $u$ such that $u\in L^p(\Omega)$ and $|\nabla u|\in L^p(\Omega)$ ($u\in L_{\text{loc}}^p(\Omega)$ and $|\nabla u|\in L_{\text{loc}}^p(\Omega)$, respectively). By $L_0^{1, p}(\Omega)$ we denote the homogeneous Sobolev space, i.e., the space of functions 
$u \in W_{\text{loc}}^{1, p}(\Omega)$ such that $|\nabla u| \in L^p(\Omega)$, and 
$||\nabla u -\nabla \varphi_j||_{L^p(\Omega)} \to 0$ as $j \to \infty$ for a sequence $\varphi_j \in C^\infty_0(\Omega)$. 

When $1<p<n$ and $\Omega=\R^n$, we will identify $L_0^{1, p}(\R^n)$ with 
the space of all functions $u \in W_{\text{loc}}^{1, p}(\R^n)$ such that $u \in L^{\frac{np}{n-p}}(\R^n)$ 
and $|\nabla u| \in  L^p(\R^n)$. For $u \in L_0^{1, p}(\R^n)$, the norm $||u||_{1,p}$ is equivalent to 
$$
||u||_{L^{\frac{np}{n-p}}(\R^n)} + ||\nabla u||_{L^p(\R^n)}. 
$$
It is easy to see that $C^\infty_0(\R^n)$ is dense in $L_0^{1, p}(\R^n)$ with respect to 
this norm (see, e.g., \cite{MZ}, Sec. 1.3.4). 

If $1<p<n$ and $\Omega=\R^n$, then the dual Sobolev space $L^{-1, p'}(\R^n)= L_0^{1, p}(\R^n)^*$ is the  space of distributions $\nu$ such that 
$$
|| \nu ||_{-1, p'} = \sup \, \frac{\left | \langle u, \nu \rangle \right |}{||u||_{1, p}} < +\infty,
$$ 
where the supremum is taken over all $u \in L_0^{1, p}(\R^n)$, $u \not=0$. We write $\nu \in L_{{\rm loc}}^{-1, p'}(\R^n)$ if $\varphi \, \nu \in L^{-1, p'}(\R^n)$, for every $\varphi \in C^\infty_0(\R^n)$.

For $u\in W_{\text{loc}}^{1, p}(\Omega)$, we define the $p$-Laplacian $\Delta_p$ ($1<p<\infty$), in the
distributional sense, i.e., for every $\vp \in C_0^{\infty}(\Omega)$,
\begin{equation}\label{distri}
\langle  \Delta_p u,   \vp \rangle =\langle \text{div}(|\nabla u|^{p-2}\nabla u, \vp \rangle =
- \int_{\Omega}|\nabla u|^{p-2}\nabla u\cdot \nabla \vp \, dx.
\end{equation}

A \textit{finite energy} solution $u \ge 0$ to \eqref{eq1} is understood in the sense that 
$u \in L_0^{1, p} (\Omega)\cap L_{\text{loc}}^q(\Omega, d \sigma)$, and, for every  $\vp \in C_0^{\infty}(\Omega)$, 
\begin{equation}\label{q-distri}
\int_{\Omega}|\nabla u|^{p-2}\nabla u\cdot \nabla \vp \, dx = \int_\Omega u^q \vp \, d \sigma. 
\end{equation}

We need to extend the definition of solutions to $u$ not necessarily in $W_{\text{loc}}^{1, p}(\Omega)$. We will understand 
solutions in the following potential-theoretic sense using $p$-superharmonic functions, which is equivalent to the notion of 
locally renormalized solutions in terms of test functions (see \cite{KKT}).

A function $u \in W_{\text{loc}}^{1, p}(\Omega)$ is called $p$-harmonic if it satisfies the homogeneous equation
$\Delta_p u=0$. Every $p$-harmonic function has a continuous representative which coincides with $u$ a.e. 
(see \cite{HKM}).

As usual, $p$-superharmonic functions are defined via a comparison principle. We say that $u\!: \Omega \rightarrow (-\infty, \infty]$ is $p$-superharmonic if $u$ is lower semicontinuous, is not identically infinite in any component of $\Omega$, and satisfies the following comparison principle:  Whenever $D\subset\subset \Omega$ and $h \in C( \bar{D})$ is $p$-harmonic in $D$, with $h \leq u$ on $\partial D$, then $h \leq u$ in $D$. 

A $p$-superharmonic function $u$ does not necessarily belong to $\mathrm{W}^{1, p}_{\text{loc}}(\Omega)$, but its truncates $T_k(u) = \min (u,k)$ do, for all $k>0$.   In addition, $T_k(u)$ are supersolutions, i.e., $-\text{div}(|\nabla T_k(u)|^{p-2}\nabla T_k(u)) \geq 0$, in the distributional sense. We will need  the generalized gradient of a $p$-superharmonic function $u$ defined by \cite{HKM}: 
$$Du = \lim_{k\rightarrow \infty} \nabla (T_k(u)) .$$

We note that every $p$-superharmonic function $u$ has a quasicontinuous representative which coincides with $u$ quasieverywhere (q.e.), i.e., 
everywhere except for a set of $p$-capacity zero. We will assume that $u$ is always chosen this way. 

Let $u$ be $p$-superharmonic, and let $1\leq r < \frac{n}{n-1}$. Then $|Du|^{p-1}$, and hence $|Du|^{p-2}Du$, 
belong to $L^r_{\text{loc}}(\Omega)$ \cite{KM1}.  This allows us to define a nonnegative distribution $-\Delta_p u$ for each  $p$-superharmonic function $u$ by 
\begin{equation}
-\langle  \Delta_p u,   \vp \rangle = \int_{\Omega} |D u|^{p-2} D u\cdot \nabla \vp \, dx, 
\end{equation}
for all $\vp \in C^{\infty}_{0}(\Omega)$.  Then by the Riesz representation theorem there exists a unique measure $\mu[u] \in M^{+}(\Omega)$ so that $-\Delta_p u = \mu[u]$. 
\begin{defn}\label{psupersense} For a nonnegative locally finite measure $\omega$ in $\Omega$ we will say that 
$$- \Delta_p  u = \omega \quad {\rm in} \, \, \Omega$$ 
in  the potential-theoretic sense if $u$ is $p$-superharmonic in $\Omega$, and $\mu[u] = \omega$.  

Thus, $-\Delta_p u = \sigma u^q$ if $u\ge 0$ is $p$-superharmonic in $\Omega$, 
$u \in L^q_{loc} (\Omega, d \sigma)$, and $d \mu[u] = u^q \, d \sigma$. 
\end{defn}
\begin{defn}\label{psupersol} A function $u\ge 0$ is a supersolution to \eqref{eq1} if 
$u$ is $p$-superharmonic, $u \in L^q_{loc} (\Omega, d \sigma)$, and,  for every nonnegative $\vp \in C_0^{\infty}(\Omega)$, 
\begin{equation}\label{supersoleq}
\int_{\Omega} |D u|^{p-2} D u\cdot \nabla \vp \, dx \ge \int_\Omega u^q \vp \, d \sigma. 
\end{equation}
\end{defn}

Supersolutions to \eqref{eq1} in the sense of Definition~\ref{psupersol} are closely related to supersolutions associated with the integral equation  
\begin{equation} \label{integraleq}
u=\w(u^q \, d\sigma) \quad d \sigma\text{-}{\rm a.e.}, 
\end{equation}
that is, measurable functions $u\ge 0$ such that $\w(u^q d\sigma) \le u < \infty$ $d \sigma$-${\rm a.e.}$ 

We will use the following universal lower bound for supersolutions  obtained in \cite{CV}. 
%
%
\begin{theorem} \label{thmlowest}
Let $1<p<n$, $0<q<p-1$, and $\sigma \in M^+(\R^n)$. Suppose $u$ is a nontrivial $p$-superharmonic supersolution 
to \eqref{eq1}.  Then the inequality 
\begin{equation}\label{lowerest}
u \ge C \,\big( \w\sigma\big )^{\frac{p-1}{p-1-q}} 
\end{equation}
 holds, where $C$ is a positive constant depending only on $p, q$, and $n$.  
 
 The same lower bound holds for a 
nontrivial supersolution  to the integral equation \eqref{integraleq}. 
 If $p\ge n$, there is only a trivial supersolution $u=0$ on $\R^n$. 
\end{theorem}

We will employ some fundamental results of the potential theory of quasilinear elliptic equations.  
The following important weak continuity result \cite{TW1} will be used to prove the existence of $p$-superharmonic solutions to quasilinear equations.
\begin{theorem} \label{weakconv} Suppose $\{u_n\}$  is a sequence of nonnegative 
$p$-super\-harmonic functions that converges a.e. to a 
$p$-superharmonic function $u$ in an open set $\Omega$. Then $\mu[u_n]$ converges weakly to $\mu[u]$, i.e.,
 for all $\vp \in C_0^{\infty}(\Omega)$, 
$$\lim_{n \to \infty}\int_{\Omega} \vp \, d \mu[u_n] = \int_{\Omega}\vp \, d \mu[u].$$
\end{theorem}

The next result  \cite{KM2} is concerned with global pointwise estimates of nonnegative $p$-superharmonic functions in terms 
of Wolff's potentials discussed in the Introduction.  
\begin{theorem} \label{thmpotest} Let $1<p\le n$. Let $u$ be a $p$-superharmonic function in $\R^n$ with $\inf_{\R^n}u=0$. 
If $\omega$ is a nonnegative Borel  measure in $\R^n$ such that $-\Delta_pu=\omega$, then 
$$ \frac{1}{K} \, \w\omega (x) \le u(x) \le K \, \w\omega(x), \quad x \in \R^n,$$
where $K$ is a positive constant depending only on $n, p.$
\end{theorem}

The following theorem is due to Brezis and Browder \cite{BB} (see also 
\cite{MZ}, Theorem 2.39). 
\begin{theorem} \label{brezisbrowder} Let $1<p<n$. Suppose $u \in L^{1,p}_0(\R^n)$,  and $\mu \in M^+(\R^n) \cap L^{-1,p'}(\R^n)$. 
Then $u \in L^1 (\R^n, \mu)$ (for a  quasicontinuous representative of $u$), and 
\begin{equation} \label{Rieszpote}
\langle \mu, u\rangle = \int_{\R^n} u \, d \mu. 
\end{equation}  
\end{theorem}
We observe that if, under the assumptions of this theorem, $-\Delta_pu = \mu$, then it follows 
(see \cite{MZ}, Theorem 2.34)  
\begin{equation} \label{Rieszpote1}
\langle \mu, u\rangle = \int_{\R^n} u \, d \mu = ||u||^p_{1,p} = ||\mu||_{-1,p'}^{p'}. 
\end{equation}

For $0< \alpha <n$ and $\sigma \in M^+(\R^n)$, the Riesz potential of $\sigma$ is defined by
\begin{equation} \label{Rieszpote2}
\I_{\alpha}\sigma(x)=\int_0^{\infty} \frac{\sigma(B(x, r))}{r^{n-\alpha}}\frac{dr}{r} =
 \frac{1}{n-\alpha} \int_{\R^n} \frac{d \sigma(y)}{|x-y|^{n-\alpha}}, \quad x \in \R^n. 
\end{equation}

For $1<p<\infty$ and $0<\alpha<\frac{n}{p}$, the Wolff potential of order $\alpha$ is defined by 
$$
\wa\sigma(x)= \int_0^\infty\left( \frac{\sigma(B(x,s))}{s^{n-\alpha p}}\right)^{\frac{1}{p-1}}\frac{ds}{s}, \quad x \in \R^n.
$$
Note that ${\bf W}_{\alpha,2} \sigma = I_{2 \alpha} \sigma$ if $0<\alpha < \frac n 2$. In particular, ${\bf W}_{1,2} \sigma = 
I_2 \sigma$ is the Newtonian potential for $n \ge 3$. 

We will need the following Wolff's inequality \cite{HW} (see also \cite{AH}, Sec. 4.5) 
which gives precise estimates of the energy associated with the Wolff potential:
\begin{theorem} \label{thmwolffineq}  Suppose $1<p<\infty$, $0<\alpha<\frac n p$, and $\sigma \in M^+(\R^n)$. 
Then there exists a constant $C>0$ depending only on $p, \alpha$, and $n$ such that 
\begin{equation} \label{wolffinequality}
\frac 1 C \, \int_{\R^n} (I_\alpha \sigma)^{p'} \, dx \le  \int_{\R^n} {\bf W}_{\alpha,p} \sigma \, d \sigma \le C  \, \int_{\R^n} (I_\alpha \sigma)^{p'} \, dx, 
\end{equation}
where $\frac 1 p + \frac {1}{p'}=1$. 
\end{theorem}

\section{Existence and minimality of finite energy solutions} \label{energy}

In this section, we deduce a necessary and sufficient condition for the existence of a finite energy  solution, and construct a  minimal solution  to \eqref{eq1}. We will assume that $1<p<n$, since for $p\ge n$ there are only trivial nonnegative supersolutions on $\R^n$ (Theorem~\ref{thmlowest}; see also \cite{HKM}, Theorem 3.53).

\begin{lemma}\label{neceFES} Suppose there exists a nontrivial supersolution $u \ge 0$,  
$u \in L_0^{1,p}(\R^n)\cap L^q_{{\rm loc}}(\R^n, d \sigma)$ to \eqref{eq1}. 
Then $$-\Delta_p u \in L^{-1, p'}(\R^n) \cap M^+(\R^n).$$ 
Moreover,  $u \in L^{1+q}(\R^n, \sigma)$ (for a 
quasicontinuous  representative of $u$), and condition  \eqref{wolffintegral} holds. 
\end{lemma}

\begin{proof} Suppose 
$u \in L_0^{1,p}(\R^n)\cap L^q_{{\rm loc}}(\R^n, d \sigma)$ is a supersolution  to \eqref{eq1}. 
Then by H\"older's inequality,  for every $\vp \in C^\infty_0(\R^n)$, 
$$
|\langle \Delta_p u,  \vp\rangle| =\left | \int_{\R^n} |\nabla u|^{p-2}\nabla u \cdot \nabla \vp \, dx\right | 
\le ||\nabla u||_{L^p(\R^n)}^{p-1} ||\nabla \vp||_{L^p(\R^n)}. 
$$
Hence, $\Delta_p u \in L^{-1,p'}(\R^n)$. If $\vp \ge 0$, then 
$$
- \langle \Delta_p u,  \vp\rangle = \int_{\R^n} |\nabla u|^{p-2}\nabla u \cdot \nabla \vp \, dx \ge 
\int_{R^n} \vp \, u^q  \, d \sigma \ge 0, 
$$
and consequently $-\Delta_p u \in M^+(\R^n)$. 

It follows that  $d \mu = u^q \, d \sigma \in M^+(\R^n)\cap L^{-1,p'}(\R^n)$. 
Let  $\{\vp_j\}$ be a sequence of nonnegative $C^\infty_0$-functions  such that 
$\vp_j \to u$ in $L^{1,p}_0 (\R^n)$. By definition, 
$$
\int_{\R^n}|\nabla u|^{p-2}\nabla u \cdot \nabla \vp_j \, dx \ge  \langle \mu, \vp_j \rangle. 
$$
 Hence, 
$$ \int_{\R^n} |\nabla u|^p \, dx = \lim_{j \to \infty}\int_{\R^n}|\nabla u|^{p-2}\nabla u \cdot \nabla \vp_j \, dx \ge  \lim_{j \to \infty} \langle \mu, \vp_j\rangle 
= \langle \mu, u \rangle.$$
Let us assume as usual that $u$ coincides with its quasicontinuous  representative.  
Then, applying  Theorem \ref{brezisbrowder}, we deduce 
$$
\langle \mu, u \rangle = \int_{\R^n} u \, d \mu = \int_{\R^n} u^{1+q} \, d \sigma < \infty. 
$$
By Theorem \ref{thmlowest}, it follows that if $u \not\equiv 0$, then $u \ge C \,\big( \w\sigma\big )^{\frac{p-1}{p-1-q}}$, and consequently  
\eqref{wolffintegral} holds. 
\end{proof}

\begin{lemma} \label{Wolfflemma} For every $r>0$, 
\begin{equation}\label{Wolffest}
\wa((\wa\sigma)^{r}d\sigma)(x) \ge C (\wa\sigma(x))^{\frac{r}{p-1}+1}, \quad x \in \R^n, 
\end{equation}
where $C$ depends only on $p$, $q$, $r$, $\alpha$, and $n$. 
\end{lemma}
\begin{proof}
For $t>0$, obviously, 
$$\wa\sigma(y)= \int_0^t\left( \frac{\sigma(B(y,s))}{s^{n-\alpha p}}\right)^{\frac{1}{p-1}}\frac{ds}{s} + \int_t^{\infty}\left( \frac{\sigma(B(y,s))}{s^{n-\alpha p}}\right)^{\frac{1}{p-1}}\frac{ds}{s}$$ 
For $y \in B(x,t)$, we have
$$\int_t^{\infty}\left( \frac{\sigma(B(y,s))}{s^{n-\alpha p}}\right)^{\frac{1}{p-1}}\frac{ds}{s}
=\int_{t/2}^{\infty}\left( \frac{\sigma(B(y,2r))}{(2r)^{n-\alpha p}}\right)^{\frac{1}{p-1}}\frac{dr}{r}$$
$$=\left(\frac{1}{2}\right)^{\frac{n-\alpha p}{p-1}}\int_{t/2}^{\infty}\left( \frac{\sigma(B(y,2s))}{s^{n-\alpha p}}\right)^{\frac{1}{p-1}}\frac{ds}{s} \ge C_{n,p,\alpha}\int_t^{\infty}\left( \frac{\sigma(B(y,2s))}{s^{n-\alpha p}}\right)^{\frac{1}{p-1}}\frac{ds}{s},$$ 
where $C_{n,p,\alpha}=\left(\frac{1}{2}\right)^{\frac{n-\alpha p}{p-1}}.$ 
Since $s \ge t$ and $ y \in B(x,t)$, then $B(y,2s) \supset B(x,s)$, which implies 
\begin{equation} \label{ballest}
\wa\sigma(y) \ge C_{n,p,\alpha}\int_t^{\infty}\left( \frac{\sigma(B(x,s))}{s^{n-\alpha p}}\right)^{\frac{1}{p-1}}\frac{ds}{s}.  
\end{equation} 
Notice that  $$\wa((\wa\sigma)^{r}d\sigma)(x)= \int_0^{\infty}\left( \frac{\int_{B(x,t)}[\wa\sigma(y)]^rd\sigma(y)}{t^{n-\alpha p}}\right)^{\frac{1}{p-1}}\frac{dt}{t}.
$$
By \eqref{ballest}, we obtain 
$$\wa((\wa\sigma)^{r}d\sigma)(x)\ge $$
$$\ge \int_0^{\infty}\left( \frac{\int_{B(x,t)}\Bigl[C_{n,p,\alpha}\int_t^{\infty}\left( \frac{\sigma(B(x,s))}{s^{n-\alpha p}}\right)^{\frac{1}{p-1}}\frac{ds}{s}\Bigr]^rd\sigma(y)}{t^{n-\alpha p}}\right)^{\frac{1}{p-1}}\frac{dt}{t} 
$$
$$\ge C_{n,p,\alpha}^{\frac{r}{p-1}} \int_0^{\infty} \Bigl[\int_t^{\infty}\left( \frac{\sigma(B(x,s))}{s^{n-\alpha p}}\right)^{\frac{1}{p-1}}\frac{ds}{s}\Bigr]^{\frac{r}{p-1}}\left(\frac{\sigma(B(x,t))}{t^{n-\alpha p}}\right)^{\frac{1}{p-1}}\frac{dt}{t}.
$$
Integrating by parts, we deduce  
$$\wa((\wa\sigma)^{r}d\sigma)(x)\ge  \frac{C_{n,p,\alpha}^{\frac{r}{p-1}}}{\frac{r}{p-1}+1}\left(\int_0^{\infty}\left( \frac{\sigma(B(x,s))}{s^{n-\alpha p}}\right)^{\frac{1}{p-1}}\frac{ds}{s}\right)^{\frac{r}{p-1}+1}.$$
Thus, 
$$
\wa((\wa\sigma)^{r} d\sigma)(x)\ge  C_{n,p,\alpha,r}\left(\w\sigma(x)\right)^{\frac{r}{p-1}+1}.
$$
\end{proof}

Setting $r = \frac{q(p-1)}{p-1-q}$ in Lemma \ref{Wolfflemma}, we deduce
\begin{equation} \label{wolffestimate}
\wa((\wa\sigma)^{\frac{q(p-1)}{p-1-q}}d\sigma)(x)\ge  \kappa \, \left(\wa\sigma(x)\right)^{\frac{p-1}{p-1-q}},
\end{equation}
where $\kappa$ depends only on $p$, $q$, and $n$. 

 Let us define a nonlinear integral operator $T$ by 
\begin{equation} \label{T}
T(f)(x) = \Big (\wa(f d\sigma)\Big)^{p-1}(x), \quad x \in \R^n.
\end{equation}
\begin{lemma}\label{Top} Let $1<p<\infty$, $0<\alpha<n$, and $0<q<p-1$. 
Suppose 
\begin{equation}\label{T-bound}
\int_{\R^n}\left(\wa\sigma\right)^{\frac{(1+q)(p-1)}{p-1-q}}d\sigma < \infty. 
\end{equation}  
Then $T$ is a bounded operator from 
$L^{\frac{1+q}{q}}(\R^n, d\sigma)$ to $L^{\frac{1+q}{p-1}}(\R^n, d\sigma).$
\end{lemma}
\begin{proof}
Clearly, 
\begin{equation*} \label{wolnorm}
||(\wa(fd\sigma))^{p-1}||_{L^{\frac{1 +q}{p-1}}(d\sigma)}=\left(\int_{\R^n}\Big(\wa(fd\sigma)\Big)^{1+q}d\sigma \right)^{\frac{p-1}{1+q}}. 
\end{equation*}
We have  $$\wa(fd\sigma)(x) \le \int_0^{\infty} \left( \frac{\sigma(B(x,r))}{r^{n-\alpha p}}\right)^{p'-1}M_{\sigma}f(x)^{p'-1}\frac{dr}{r}= M_{\sigma}f(x)^{p'-1}\w\sigma(x),$$
where the centered maximal operator  $M_{\sigma}$ is defined by 
$$M_{\sigma}f(x) =\sup_{r>0} \frac{1}{\sigma(B(x,r))} \int_{B(x,r)}|f| \, d\sigma, \quad x \in \R^n.$$
It is well known  that $M_\sigma: L^s (\R^n, d \sigma) \to L^s (\R^n, d \sigma)$ is a bounded operator for all $s>1$. Let $s=\frac{1+q}{q}$. 
Then, using H\"older's inequality with the exponents $\beta= \frac{p-1}{q} > 1$ and $\beta '= \frac{p-1}{p-1-q}$, we estimate, 
$$\int_{\R^n}\Big(\wa(fd\sigma)\Big)^{1+q}d\sigma \le \int_{\R^n}(M_{\sigma}f)^{\frac{1+q}{p-1}}(\wa\sigma)^{1+q} d\sigma$$
$$\le  \left(\int_{\R^n}(M_{\sigma}f)^{\frac{1+q}{q}}d\sigma\right)^{\frac{q}{p-1}}\left(\int_{\R^n}(\wa\sigma)^{\frac{(1+q)(p-1)}{p-1-q}}d\sigma\right)^{\frac{p-1-q}{p-1}}$$
$$\le C\left(\int_{\R^n}f^{\frac{1+q}{q}}d\sigma\right)^{\frac{q}{p-1}}\left(\int_{\R^n}(\wa\sigma)^{\frac{(1+q)(p-1)}{p-1-q}}d\sigma\right)^{\frac{p-1-q}{p-1}}.$$
Thus, 
$$ ||\wa(fd\sigma)^{p-1}||_{L^{\frac{1+q}{p-1}}(d\sigma)}\le c \, ||f||_{L^{\frac{1+q}{q}}(d\sigma)}.$$
\end{proof}
\begin{rmk}
It is not difficult to see that actually  \eqref{T-bound} is also necessary  for the boundedness of the operator $T\!: L^{\frac{1+q}{q}}(\R^n, d\sigma) \to L^{\frac{1+q}{p-1}}(\R^n, d\sigma)$  (see, for example, \cite{COV}).
\end{rmk}

\begin{theorem}\label{integralsol}
Let $1<p<n$, and $0<q<p-1$. Suppose that condition \eqref{wolffintegral} holds. Then there exists  a solution 
$u  \in L^{1+q}(\R^n, d\sigma)$ to the integral equation \eqref{integraleq}. 
\end{theorem}

\begin{proof}
By Lemma \ref{Top}, we have, for all $f \in L^{\frac{1+q}{q}}(\R^n,d\sigma$),
\begin{equation} \label{Wolffbdd}
\int_{\R^n}  \Big(\w(fd\sigma)\Big)^{1+q}\, d\sigma \le C \left( \int_{\R^n}f^{\frac{1+q}{q}} \, d\sigma\right)^{\frac{q}{p-1}}.  
\end{equation}
Let $u_0=c_0 \left(\w \sigma \right)^{\frac{p-1}{p-1-q}}$, where $c_0 > 0$ is a small constant to be chosen later on. We construct 
a sequence of iterations $u_j$ as follows: 
\begin{equation} \label{iterations}
u_{j+1}=\w(u^q_jd\sigma), \quad j=0,1,2, \ldots .
\end{equation}
Applying Lemma \ref{Wolfflemma}, we have 
$$u_1=\w(u_0^qd\sigma)=c_0^{\frac{q}{p-1}} \,  \w(\left(\w\sigma\right)^{\frac{q(p-1)}{p-1-q}}d\sigma) \ge 
c_0^{\frac{q}{p-1}} \,  \kappa \,  \left(\w\sigma\right)^{\frac{p-1}{p-1-q}},$$ where $\kappa$ is the constant in \eqref{Wolffest}. Choosing $c_0$ so that $c_0^{\frac{q}{p-1}} \,  \kappa \ge c_0$, we obtain $u_1 \ge u_0.$ By induction, we can show that the sequence $\{u_j\}$ is nondecreasing. Note that $u_0 \in L^{1+q}(\R^n, d\sigma)$ by assumption. Suppose that $u_0, \ldots, u_j \in L^{1+q}(\R^n,d\sigma)$. Then
$$\int_{\R^n}u_{j+1}^{1+q} \, d\sigma= \int_{\R^n}(\w(u_j^q \, d\sigma)^{1+q}d\sigma.$$
Applying \eqref{Wolffbdd} with $f=u_j^q$, we obtain by induction, 
\begin{equation}\label{ujintegral}
\int_{\R^n}u_{j+1}^{1+q}\, d\sigma \le C \left(\int_{\R^n}u_j^{1+q} \, d\sigma \right)^{\frac{q}{p-1}} < \infty. 
\end{equation}
Since $u_j \le u_{j+1}$, the preceding inequality yields 
$$\int_{\R^n}u_{j+1}^{1+q}\, d\sigma \le C \left(\int_{\R^n}u_{j+1}^{1+q} \, d\sigma \right)^{\frac{q}{p-1}} < \infty.$$
Thus, 
$$\left(\int_{\R^n}u_{j+1}^{1+q}d\sigma \right)^{\frac{p-1-q}{p-1}} \le  C < \infty.$$
Using the Monotone Covergence Theorem and passing to the limit as $j \to\infty$ in \eqref{iterations}, we see that there exists $u= \lim_{j \to \infty} u_j$, such that $u \in L^{1+q}(\R^n,d\sigma)$, and the integral equation \eqref{integraleq} holds. 
\end{proof}

\begin{lemma} \label{duallemma} Let $u \in L^{1+q}(\R^n, d\sigma)$ be a nonnegative supersolution to the integral equation \eqref{integraleq}. 
 Then 
\begin{equation} \label{dual}
u^q \, d\sigma \in L^{-1, p'}(\R^n). 
\end{equation}
\end{lemma}
\begin{proof}
Let $d\nu =u^q \, d\sigma$. We need to show that, for all $\vp \in \Cnot$, 
\begin{equation} \label{sobolev}
\left | \int_{\R^n}\vp \, d\nu\right | \le c \left(\int_{\R^n} |\nabla \vp|^p \, dx\right)^{\frac{1}{p}}.
\end{equation}
It is easy to see that the above inequality is equivalent to  
\begin{equation} \label{traceine}
\left |\int_{\R^n}\re g \, \, d \nu \right| \le c \,  \left(\int_{\R^n} |g|^p \, dx \right)^{\frac{1}{p}}, 
\end{equation}
 for all $g \in L^p(\R^n)$, where $\re g$ is the Riesz potential of g of order $1$. 
By duality, \eqref{traceine} is equivalent to 
\begin{equation} \label{finiteenergy}
\int_{\R^n}(\re  \nu )^{p'} \, dx < \infty.
\end{equation}
Using Wolff's inequality \eqref{wolffinequality}, we deduce that \eqref{finiteenergy} holds if and only if 
\begin{equation} \label{wolffcond}
\int_{\R^n}\w\nu \, d\nu  < \infty.
\end{equation}
Notice that since $u\ge \w(u^q \, d \sigma)$ and $u \in L^{1+q}(\R^n, d \sigma)$ then 
$$\int_{\R^n}\w\nu \, d \nu = \int_{\R^n}\w(u^q \, d \sigma) \, u^q \, d \sigma \le  \int_{\R^n}u^{1+q} \, d \sigma < \infty.$$
Thus,  \eqref{finiteenergy} holds. This completes the proof of the lemma.
\end{proof}

We will need a weak comparison principle which goes back to 
P. Tolksdorf's work on quasilinear equations  (see, e.g.,  \cite{PV1}, Lemma 6.9, in the case of renormalized solutions in bounded domains). 

\begin{lemma}  \label{wcp} 
Suppose $\mu, \omega \in M^+(\R^n)\cap L^{-1, p'}(\R^n)$.
 Suppose $u$ and $v$ are   (quasicontinuous) solutions 
 in $L_0^{1, p}(\R^n)$ of the equations $-\Delta_p u = \mu$ and $-\Delta_p v= \omega$, respectively.  If $\mu \le \omega$,  then $u \le v $ q.e.  
\end{lemma}
\begin{proof}
For every $\varphi \in L_0^{1, p}(\R^n)$, we have by Theorem~\ref{brezisbrowder}, 
\begin{equation} \label{umu}
\int_{\R^n} |\nabla u|^{p-2} \nabla u \cdot \nabla \varphi \,dx = \langle \mu, \varphi\rangle=\int_{\R^n} \varphi \, d \mu, 
\end{equation}
\begin{equation} \label{vomega}
\int_{\R^n} |\nabla v|^{p-2} \nabla v \cdot \nabla \varphi \,dx= \langle \omega, \varphi\rangle= \int_{\R^n} \varphi \, d \omega.
\end{equation}
Hence,
\begin{equation} \label{subtract}
\int_{\R^n} (|\nabla u|^{p-2} \nabla u - |\nabla v|^{p-2} \nabla v )\cdot \nabla \varphi \,dx = \int_{\R^n} \varphi \, d \mu - \int_{\R^n} \varphi \, d \omega.
\end{equation}
Since $\mu \le \omega$, it follows that, for every $\varphi \in L_0^{1, p}(\R^n)$, $\varphi \ge 0$, we have 
\begin{equation}\label{inequality}
\int_{\R^n} (|\nabla u|^{p-2} \nabla u - |\nabla v|^{p-2} \nabla v )\cdot \nabla \varphi \,dx \le 0.
\end{equation}

Testing \eqref{inequality} with $ \varphi = (u-v)^{+}=\max\{u-v,0\} \in L_0^{1, p}(\R^n)$, we obtain,
$$I=\int_{\R^n} (|\nabla u|^{p-2} \nabla u - |\nabla v|^{p-2} \nabla v )\cdot \nabla (u-v)^{+} \,dx \le 0.$$
Let $A=\{x \in \R^n : u(x) > v(x) \}$, then 

$$I=\int_{A} (|\nabla u|^{p-2} \nabla u - |\nabla v|^{p-2} \nabla v )\cdot \nabla (u-v) \,dx \le 0.$$

Note that $$(|\nabla u|^{p-2} \nabla u - |\nabla v|^{p-2} \nabla v )\cdot \nabla (u-v) \ge   0.$$
Thus,
$$0\le \int_{A} (|\nabla u|^{p-2} \nabla u - |\nabla v|^{p-2} \nabla v )\cdot \nabla (u-v) \,dx  
= \int_{A}  \varphi (d \mu -d \omega)\le 0.$$
It follows that  $\nabla (u-v) =0$ a.e. 
on $A$. By Lemma 2.22 in \cite{MZ}, for every $a>0$, 
$$
\text{cap}_p \, \{u-v>a\} \le \frac{1}{a^p} \int_{A} |\nabla (u-v)|^p \, dx  =0. 
$$
Consequently, $\text{cap}_p(A)=0$, i.e., $u \le v$ q.e.
\end{proof}

We are now in a position to prove the main theorem of this section. 

\begin{theorem} \label{energysol} Let $1<p< n$ and $0<q<p-1$. Let $\sigma \in M^+(\R^n)$, $\sigma \not=0$. 
Suppose that \eqref{wolffintegral} holds.  
Then there exists a nontrivial solution $w\in L_0^{1, p}(\R^n)\cap L^q_{\rm loc}(\R^n, d \sigma)$ to \eqref{eq1}.  Moreover, $w$ is a minimal solution, i.e., $w \le u$ $d \sigma$-a.e.  (q.e. for quasicontinuous representatives) 
for any nontrivial solution $u\in L_0^{1, p}(\R^n)\cap L^q_{\rm loc}(\R^n, d \sigma)$ to \eqref{eq1}. \end{theorem}

\begin{proof} We first show that there exists a solution $w\in L_0^{1, p}(\R^n)\cap L^q_{\rm loc}(\R^n, d \sigma)$ to \eqref{eq1}. Applying Theorem \ref{integralsol}, we conclude that there exists a solution $v \in L^{1+q}(\R^n, d\sigma)$ to the integral equation \eqref{integraleq}. By using a constant multiple $c \, v$ in place of $v,$ we can assume that $v = K\w(v^q \, d\sigma)$, where $K$ is the constant in Theorem \ref{thmpotest}.  Then by  Lemma \ref{duallemma} and Theorem \ref{thmlowest},  
$$
v^q \, d \sigma \in L^{-1, p'}(\R^n), \quad \text{and} \quad v \ge C  \, K^{\frac{p-1}{p-1-q}} \left(\w\sigma\right)^{\frac{p-1}{p-1-q}}, 
$$
where $C$ is the constant in \eqref{lowerest}.

We set $$w_0= c_0 \left(\w\sigma\right)^{\frac{p-1}{p-1-q}}, \quad d\omega_0=w_0^q \, d\sigma, $$ 
where $c_0>0$ is a small constant to be determined later. 
Since $$w_0 \le \frac {c_0}{CK^{\frac{p-1}{p-1-q}}} \, v ,$$ it follows 
that, for $c_0\le C K^{\frac{p-1}{p-1-q}}$, we have $w_0\le v$. Hence,  
$$ w_0 \in L^{1+q} (\R^n, d \sigma), \quad \text{ and} \quad \omega_0  \in  L^{-1, p'}(\R^n) .$$
Then there exists a unique nonnegative solution $w_1 \in L_0^{1,p}(\R^n)$ 
to the equation 
$$-\Delta_p w_1 = \omega_0,  \quad \text{and} \quad  ||w_1||_{1,p}^{p-1} = ||\omega_0||_{-1, p'}.$$ 
(See \eqref{Rieszpote1}.) Moreover, by Theorem~\ref{thmpotest}, 
$$0 \le w_1 \le K \, \w \omega_0 \le K \, \w (v^q d\sigma) =  v.$$
Consequently, by Lemma~\ref{duallemma}, 
$$ w_1 \in L^{1+q} (\R^n, d \sigma),  \quad \text{and} \quad w_1^q \, d \sigma \in  L^{-1, p'}(\R^n).$$
We deduce, using \eqref{wolffestimate}, 
$$w_1 \ge \frac{1}{K}   \w \omega_0=\frac{c_0^{\frac{q}{p-1}}}{K} 
\,  \w \Big ( \left(\w\sigma\right)^{\frac{q(p-1)}{p-1-q}} d \sigma\Big)
$$
$$ \ge \frac{c_0^{\frac{q}{p-1}} \kappa}{K} 
\,  \left(\w\sigma\right)^{\frac{p-1}{p-1-q}} =   \frac{c_0^{\frac{q}{p-1}-1} \kappa}{K} \, w_0.
$$
Hence, for $c_0\le  (K^{-1} \, \kappa)^{\frac{p-1}{p-1-q}}$, we have $v \ge w_1\ge w_0$. 

To prove the minimality of $w$, we will need $c_0 \le C$, so we pick $c_0$ so that  
\begin{equation}\label{c0}
0<c_0\le \min\big \{C \, K^{\frac{p-1}{p-1-q}}, \, (K^{-1} \, \kappa)^{\frac{p-1}{p-1-q}}, \, C \big \}.
\end{equation}

Let us now construct by induction a sequence $\{w_j\}_{j \ge 1}$ so that 
\begin{equation}\label{wj}
\left\{ \begin{array}{ll}
-\Delta_pw_{j}=\sigma \, w_{j-1}^q \, \,  \text{ in ~~}\,  \mathbb{R}^n,  \quad  
w_j \in L_0^{1, p}(\R^n) \cap L^{1+q}(\R^n, d \sigma),
 \\
 0\le w_{j-1} \le w_{j} \le v, \text{ q.e.,}  \quad w_{j-1}^q \, d \sigma \in  L^{-1, p'}(\R^n), 
\end{array} \right.
\end{equation}
where $\sup_j \, ||w_j||_{1,p} < \infty.$ We set $d\omega_j=w_j^q \, d\sigma$, so that
$$-\Delta_p w_j = \omega_{j-1}, \quad j=1,2, \ldots . $$

Suppose that $w_0, w_1, \ldots, w_{j-1}$ have been constructed. As in the case $j=1$, we see that, since $\omega_{j-1} \in  L^{-1, p'}(\R^n)$, 
there exists a unique $w_j \in L_0^{1, p}(\R^n)$ such that $-\Delta_p w_{j}=\omega_{j-1}$, and by \eqref{Rieszpote1}, 
$$ ||w_j||_{1,p}^{p} = ||\omega_{j-1}||^{p'}_{-1, p'} = \int_{\R^n} w_j \, w^q_{j-1} \, d \sigma.$$
By Theorem \ref{thmpotest}, we get
$$w_j \le  K \,  \w \omega_{j-1} = K \, \w(w^q_{j-1}d\sigma).$$
Using the inequality $w_{j-1} \le v$, we see that 
$$w_j \le K \, \w(v^q \, d\sigma) = v.$$
Combining these estimates, we obtain  
$$
 ||w_j||_{1,p}^{p} = \int_{\R^n} w_j \, w^q_{j-1} \, d \sigma \le \int_{\R^n} v^{1+q} \, d \sigma<\infty.
$$

Consequently,  $\{w_j\}$ is a bounded sequence in $L_0^{1,p}(\R^n)$. Notice 
that $w_{j-1} \le w_j$ by the weak comparison principle (Lemma \ref{wcp}), since $\omega_{j-2} \le 
\omega_{j-1}$ , for $j \ge 2$. 

Thus, the sequence \eqref{wj} has been constructed. Letting $w=\lim_{j \to \infty} \, w_j$, and applying the weak continuity of the $p$-Laplace  operator (Theorem~\ref{weakconv}), the Monotone Convergence Theorem, and Lemma 1.33 in \cite{HKM}, we deduce the existence of a nontrivial  solution $w \in L_0^{1,p}(\R^n)$ to  \eqref{eq1}. 

We now prove the minimality of $w$.  Suppose  $u \in L_0^{1, p}(\R^n)\cap L^q_{\rm loc}(\R^n, d \sigma)$  is any nontrivial solution to \eqref{eq1}. 
Letting $d \mu =u^q \, d \sigma$, we have $u \in L^{1+q}(\R^n, d \sigma)$, and $\mu \in   L^{-1, p'}(\R^n)$ by 
Lemma~\ref{neceFES}. To show that $ u \ge w $, notice that by Theorem \ref{thmlowest}, 
$$u \ge C \, (\w\sigma)^{\frac{p-1}{p-1-q}},$$
where $C$ is the constant in \eqref{lowerest}. By the choice of $c_0$ in \eqref{c0}, we have $w_0 \le u$, so that  $\omega_0 \le \mu.$ Therefore, by the weak comparison principle $w_1 \le u $ q.e. Arguing by induction as above, we see that $w_{j-1}\le w_j \le u $ q.e. for $ j \ge 1$. It 
follows that $ \lim_{ j \to \infty} w_j =w \le u$ q.e., which proves that $w$ is a minimal solution.  \end{proof}

By combining Lemma~\ref{neceFES} and Theorem~\ref{energysol} we conclude the proof of the 
existence part of the Theorem stated in the Introduction. In Sec.~\ref{unique} below we will establish   
 the uniqueness part using the existence of a minimal solution constructed in Theorem~\ref{energysol}. 

\section{$\Ac$-Laplace operators}\label{alaplace} 

Let us assume that $\Ac : \R^n \times \R^n \to \R^n$ satisfies the following structural assumptions: 
$$ x \to \Ac(x,\xi) \quad \text{ is measurable for all }\,  \xi \in \R^n, $$
$$ \xi \to \Ac(x,\xi) \quad \text{ is continuous for a.e.  } \, x \in \R^n,$$
and there are constants $0 < \alpha  \le \beta < \infty,$ such that for a.e. $x$ in $\R^n$, and for all $\xi$ in $\R^n$,
$$\Ac(x,\xi)\cdot \xi \ge \alpha|\xi|^p, \quad |\Ac(x,\xi)|\le \beta |\xi|^{p-1},$$
$$(\Ac(x,\xi_1) - \Ac(x,\xi_2))\cdot(\xi_1 - \xi_2) > 0 \quad \text{ if } \,\, \xi_1 \neq \xi_2,$$
$$\Ac(x,\lambda\xi)=\lambda|\lambda|^{p-2}\Ac(x,\xi),  \quad \text{ if } \, \lambda \in \R \backslash\{0\}.$$

 Consider the equation 
\begin{equation}\label{renormeq}
-\text{div}\Ac(x,\nabla u)=\mu  \quad \text{ in } \,\, \Omega,    
\end{equation}
where $\mu \in M^+(\Omega)$, and $\Omega \subseteq \R^n$ is an open set. Let us use the decomposition 
 $\mu=\mu_0 + \mu_s,$ where $\mu_0$ is absolutely continuous with respect to the $p$-capacity and  $\mu_s$ is singular with respect to the $p$-capacity. Let $T_k(s)=\max\{-k,\min\{k,s\}\}$. We say that $u$ is a {\it local renomalized solution } to \eqref{renormeq} if, for all $k>0$,  $T_k(u) \in W_{\text{loc}}^{1,p}(\Omega)$, $u \in L_{\text{loc}}^{(p-1)s}$ for $1\le s<\frac{n}{n-p},$ $Du \in L_{\text{loc}}^{(p-1)r}(\Omega)$ for $1 \le r < \frac{n}{n-1}$, and
$$\int_{\Omega} \langle\Ac(x,Du), Du\rangle \, h'(u) \, \phi \, dx + \int_{\Omega}\langle\Ac(x,Du), \nabla \phi \rangle h(u) \, \phi \, dx $$
$$=\int_{\Omega}h(u) \, \phi \,d\mu_0 +h(+\infty) \, \int_{\Omega} \phi\,d\mu_s,$$
for all $\phi \in C_0^{\infty}(\Omega) $ and $h \in W^{1,\infty}(\R)$ such that $h'$ is compactly supported; here 
$h(+\infty)=\lim_{t \to +\infty}h(t).$

In \cite{KKT}, it is shown that  every $\Ac$-superhamonic function is locally a renormalized solution, and conversely, every local renormalized solution has an $\Ac$-superharmonic representative. Consequently, 
we can work either with local renormalized solutions, or equivalently with potential theoretic solutions, 
or finite energy solutions in the case $u \in L^{1, p}_0(\Omega)$. We note that, for finite energy solutions, 
$Du$ coincides with the distributional gradient $\nabla u$, and $d \mu = u^q \, d \sigma$ is absolutely 
continuous with respect to the $p$-capacity as was mentioned above.

It is known that basic facts of potential theory stated in Sec.~\ref{preli}, including Wolff's potential estimates \cite{KM2},  and the weak continuity 
principle \cite{TW1}, remain true for the  $\Ac$-Laplacian.   From the above results it follows that our methods work, with obvious modifications, not only for the $p$-Laplace operator, but for the general $\Ac$-Laplace operator 
${\rm div} \, \Ac(x,\nabla u)$ as well. In particular, the following more general theorem holds.

\begin{theorem} Under the above assumptions on $\Ac(x,\, \xi)$, together with the conditions of the Theorem stated in Sec.~\ref{intro}, the equation 
$$
-{\rm div} \, \Ac(x,\nabla u) = \sigma \, u^q  
$$
has a solution $u \in L^{1, p}_0(\R^n)\cap L^q_{{\rm loc}} (\R^n, d \sigma)$ if and only if condition \eqref{wolffintegral}  holds. 
\end{theorem}

\section{Uniqueness}\label{unique}

In this section, we prove the uniqueness of finite energy solutions to \eqref{eq1}. We employ a convexity argument 
using some ideas of Kawohl \cite{Kaw} (see also \cite{BeKa}, \cite{BF}), together with the existence 
of a minimal solution established above.  

\begin{theorem} \label{uniquethm} Let $1<p< \infty$ and let $0<q<p-1$. Let $\sigma \in M^+(\R^n)$. 
Suppose that there exists a nontrivial solution $u \in L_0^{1, p}(\R^n)\cap L^q_{\rm loc}(\R^n, d \sigma)$ to \eqref{eq1}. Then such a  solution is unique.
\end{theorem}
\begin{proof}
Suppose $u, v$ are nontrivial solutions to \eqref{eq1} which lie in  $L_0^{1, p}(\R^n)\cap L^q_{\rm loc}(\R^n, \sigma)$. 
We first show that $u=v$ $d \sigma$-a.e. implies that $u=v$ as elements of $L_0^{1, p}(\R^n)$. 

Indeed, suppose that $u=v$ $d \sigma$-a.e.,   
and set $d \mu = u^q d \sigma= v^q d \sigma$, where $\mu \in M^+(\R^n)$, and  
\begin{equation} \label{mu-unique}
-\Delta_p u = - \Delta_p v = \mu.  
\end{equation}
As usual, we assume that $u, v$ are quasicontinuous representatives (see, e.g.,  \cite{HKM}, \cite{MZ}). Then 
by Lemma~\ref{neceFES},  $u, \, v \in L^{1+q}(\R^n, d \sigma)$,  and 
$$
\int_{\R^n} \mathbf{W}_{1, p}\mu \, d \mu  <+\infty. 
$$
By Wolff's inequality \eqref{wolffinequality}, this means that 
$
\mu \in L^{-1, p'}(\R^n)
$.   It is well known  (\cite{MZ}, Sec. 2.1.5)  that, for such $\mu$,  a finite energy solution to the equation $-\Delta_p u = \mu$ is unique. (See also Lemma~\ref{wcp} above.) Hence, 
from \eqref{mu-unique} we deduce 
$u=v$ q.e. and as elements of $L^{1, p}_0(\R^n)$.

We next show that if $u \ge v $ $d \sigma$-a.e. then $u = v$ $d \sigma$-a.e. By Theorem \ref{thmlowest}, it follows that  $u(x) > 0$, 
$v(x) > 0$, for all $x \in \R^n.$
Testing the equations 
\begin{equation} \label{Udist}
\int_{\R^n}|\nabla u |^{p-2} \nabla u \cdot \nabla \phi \, dx = \int_{\R^n}u^q \varphi \, d \sigma, \quad  \phi \in L_0^{1, p}(\R^n)\cap L^q_{\rm loc}(\R^n, d \sigma), 
\end{equation}
\begin{equation} \label{Vdist}
\int_{\R^n}|\nabla v |^{p-2} \nabla v \cdot \nabla \psi \, dx = \int_{\R^n}v^q \psi \, d\sigma, \quad  \psi \in L_0^{1, p}(\R^n)\cap L^q_{\rm loc}(\R^n, d \sigma), 
\end{equation}
with $\phi = u, \psi =v$, respectively, we obtain 
$$\int_{\R^n}|\nabla u |^p \, dx = \int_{\R^n}u^{1+q} \, d \sigma, \quad \int_{\R^n}|\nabla v |^p \, dx = \int_{\R^n}v^{1+q}  \, d \sigma.$$
Let $$\lambda_t(x)=\Bigl((1-t)v^p(x) + tu^p(x)\Bigr)^{\frac{1}{p}}.$$
Using convexity of the Dirichlet integral $\int_{\R^n} |\nabla u|^p \, dx$ in $u^p$ \cite{Kaw} (see also 
 the proof of  Lemma 2.1 in \cite{BF}), we estimate, for all $t \in [0, 1]$, 
$$\int_{\R^n} |\nabla\lambda_t(x)|^p dx \le (1-t) \int_{\R^n} |\nabla v |^p \, d x + t\int_{\R^n} |\nabla u |^p \, dx $$
$$=t\left(  \int_{\R^n}|\nabla u |^p \, dx - \int_{\R^n}|\nabla v |^p \, dx \right) + \int_{\R^n} |\nabla v |^p \, dx.$$
Thus, 
$$\int_{\R^n}\frac{|\nabla\lambda_t(x)|^p - |\nabla \lambda_0(x) |^p}{t} \,dx  \le  \int_{\R^n}u^{1+q} \, d \sigma  - \int_{\R^n}v^{1+q} \,  d \sigma.$$

Using the inequality 
$$|a|^p - |b|^p \ge p|b|^{p-2}b\cdot (a-b), \quad a, b \in \R^n,$$ 
we deduce
$$|\nabla\lambda_t|^p - |\nabla \lambda_0 |^p \ge p |\nabla \lambda_0 |^{p-2}\nabla \lambda_0 \cdot (\nabla \lambda_t-\nabla \lambda_0).$$
Notice that $\lambda_0 =v$, and consequently, for all $t \in (0, 1]$, 
\begin{equation} \label{Vequation}
 p\int_{\R^n}|\nabla v|^{p-2}\nabla v \cdot \frac{\nabla(\lambda_t -  \lambda_0 ) }{t} \,dx  \le  \int_{\R^n}u^{1+q} d\,\sigma  - \int_{\R^n}v^{1+q} d\,\sigma.
\end{equation}
Testing \eqref{Vdist} with $\psi = \lambda_t - \lambda_0 \in L_0^{1,p}(\R^n)$, we obtain
$$\int_{\R^n}|\nabla v|^{p-2}\nabla v \cdot \nabla(\lambda_t -  \lambda_0 ) \,dx  = \int_{\R^n} v^q(\lambda_t -  \lambda_0 )d\sigma.$$
Hence, by \eqref{Vequation}, for all $t \in (0, 1]$, 
\begin{equation} \label{lambdaeq}
 p\int_{\R^n} v^q \, \frac{\lambda_t -  \lambda_0 }{t} \, d\sigma \le \int_{\R^n}u^{1+q} d\,\sigma  - \int_{\R^n}v^{1+q} d\,\sigma. \end{equation}
Clearly, $ \lambda_t \ge \lambda_0$, since $ u \ge v$. Applying Fatou's lemma, we obtain 
$$\int_{\R^n} v^q \, \frac{u^p - v^p}{v^{p-1}} \, d\sigma  \le \liminf_{ t \to 0} p \, \int_{\R^n} v^q \, \frac{\lambda_t -  \lambda_0 }{t} \, d\sigma.$$

Combining this and \eqref{lambdaeq} yields
$$\int_{\R^n}  (\frac{v^qu^p}{v^{p-1}} - v^{1+q} ) \, d\sigma \le \int_{\R^n}u^{1+q} d\,\sigma  - \int_{\R^n}v^{1+q} d\,\sigma.$$
Therefore, canceling the second terms on both sides, and taking into account that $u \ge v$ $d \sigma$-a.e., 
we arrive at 
$$ 0 \ge \int_{\R^n}  (\frac{v^qu^p}{v^{p-1}} - u^{1+q} )\,d\sigma = \int_{\R^n}  \frac{v^qu^p-u^{1+q}v^{p-1}}{v^{p-1}} \,d\sigma$$
$$=\int_{\R^n}  \frac{v^qu^{1+q}(u^{p-1-q}-v^{p-1-q})}{v^{p-1}} \,d\sigma  \ge 0. $$
Hence, $u=v$ $d \sigma$-a.e.

We now complete the proof of the uniqueness property. 
Suppose that $u$ and $v$ are  nontrivial finite energy  solutions to \eqref{eq1}. Then $ \min \, (u, v) \ge w$  $d\sigma$-a.e., where $w$ is the nontrivial minimal solution constructed in Theorem \ref{energysol}. Therefore, as was shown above,  $w=u=v$ $d\sigma$-a.e., and also as 
elements of $L^{1, p}_0(\R^n)$. \end{proof}

\end{document}